\documentclass{amsart}
\usepackage{amsmath,amstext, amsthm,amssymb}

\title{Binomial Andrews-Gordon-Bressoud identities}

\author{Dennis Stanton}
\address[Dennis Stanton]{School of Mathematics, University of Minnesota, Minneapolis, MN 55455}
\email{stanton@math.umn.edu}
\date{\today}

\thanks{The author was supported by NSF grant DMS-1148634}


\newtheorem{thm}{Theorem}[section]
\newtheorem{lem}[thm]{Lemma}
\newtheorem{prop}[thm]{Proposition}
\newtheorem{cor}[thm]{Corollary}
\theoremstyle{definition}

\newtheorem{defn}[thm]{Definition}

\newtheorem{Andrques}[thm]{Andrews' Question}

\def\JV.{Josuat-Verg\`es}

\begin{document} 

\begin{abstract} Binomial versions of the Andrews-Gordon-Bressoud identities are given.
\end{abstract}

\maketitle

\section{Introduction} The Rogers-Ramanujan identities
$$
\sum_{s=0}^\infty \frac{q^{s^2}}{(q;q)_s}=\frac{1}{(q,q^4;q^5)_\infty},
\qquad 
\sum_{s=0}^\infty \frac{q^{s^2+s}}{(q;q)_s}=\frac{1}{(q^2,q^3;q^5)_\infty}
$$
where
$$
(A;q)_n=\prod_{i=0}^{n-1}(1-Aq^i), \qquad (A,B;q)_n=(A;q)_n(B;q)_n,
$$
were generalized to odd moduli at least five by the Andrews \cite{And0}. 
These identities are called the Andrews-Gordon identities
\begin{equation}
\label{AGiden}
\begin{aligned}
\sum_{ s_1\ge s_2\ge\cdots\ge s_{k}\ge 0}&
\frac{q^{s_{1}^2+\cdots +s_k^2+s_{k-r+1}+\cdots +s_k}}
{(q)_{s_1-s_2}\cdots (q)_{s_{k-1}-s_{k}}(q)_{s_{k}}}\\
&=
\frac{(q^{k+1-r},q^{k+2+r},q^{2k+3};q^{2k+3})_\infty}{(q;q)_\infty},
\qquad 0\le r\le k.
\end{aligned}
\end{equation}
(If the base $q$ is understood, we sometimes abbreviate $(A;q)_n$ as $(A)_n.$ )
 
Bressoud \cite{Bress0},\cite{Bress2} gave a version of these identities for even moduli
\begin{equation}
\label{Bresseven}
\begin{aligned}
\sum_{ s_1\ge s_2\ge\cdots\ge s_{k}\ge 0}&
\frac{q^{s_{1}^2+\cdots +s_k^2+s_{k-r+1}+\cdots +s_k}}
{(q)_{s_1-s_2}\cdots (q)_{s_{k-1}-s_{k}}(q^2;q^2)_{s_{k}}}\\
&=
\frac{(q^{k+1-r},q^{k+1+r},q^{2k+2};q^{2k+2})_\infty}{(q;q)_\infty},
\qquad 0\le r\le k.
\end{aligned}
\end{equation}

Bressoud's beautiful and efficient proof \cite{Bress1}
established both sets of identities when $r=0$. 
Moreover he had other 
closely related identities, for example,
\cite[(3.3), p. 15]{Bress2}
\begin{equation}
\label{Bress1}
\begin{aligned}
\sum_{ s_1\ge s_2\ge\cdots\ge s_{k}\ge 0}&
\frac{q^{s_1^2+\cdots +s_k^2-(s_1+\cdots +s_j)} }
{(q)_{s_1-s_2}\cdots (q)_{s_{k-1}-s_{k}}(q)_{s_{k}}}\\
=&
\sum_{s=0}^j  
\frac{(q^{k+1+j-2s},q^{k+2-j+2s},q^{2k+3};q^{2k+3})_\infty}{(q;q)_\infty},
\end{aligned} \qquad 0\le j\le k.
\end{equation}

The purpose of this paper is to examine Bressoud's proof, and develop new 
variations and generalizations of these Andrews-Gordon-Bressoud identities. 
The new results are given \S3, \S4, and \S5.

\section{The motivating question}

In \cite{And3}, which was presented by George Andrews at the May 2015 UCF meeting 
in honor of Mourad Ismail, Andrews reconsidered Bressoud's elementary proof \cite{Bress1}.  
He asked a specific question (see {Question~\ref{Qu}})
about Bressoud's proof that we answer in this section. 

We shall need a few of the relevant definitions and facts
in a recapitulation Bressoud's simple proof \cite{Bress1}. 
We shall also use these facts in later sections. 
Bressoud's key idea was to use the following Laurent polynomials, 
which have arbitrary quadratic exponents.

\begin{defn} Let
$$
H_{2n}(z,a|q)=\sum_{s=-n}^n 
\left[ \begin{matrix} 2n\\ n-s \end{matrix} \right]_q
q^{as^2} z^s.
$$
\end{defn}

Bressoud's main lemma \cite[Lemma 2]{Bress1} 
which allowed the quadratic exponent to change is next.
\begin{lem}
\label{keyBrlemma}
$$
\begin{aligned}
\frac{H_{2n}(z,a|q)}{(q;q)_{2n}}=&\sum_{s=0}^n \frac{q^{s^2}}{(q;q)_{n-s} }
\frac{H_{2s}(z,a-1|q)}{(q;q)_{2s}}.
\end{aligned}
$$
\end{lem}

This lemma may be iterated. 
\begin{prop}
$$
\frac{H_{2n}(z,a+k+1|q)}{(q;q)_{2n}}=
\sum_{n\ge s_1\ge s_2\ge\cdots\ge s_{k+1}\ge 0}
\frac{q^{s_1^2+\cdots s_{k+1}^2} H_{2s_{k+1}}(z,a|q)}
{(q)_{n-s_1}(q)_{s_1-s_2}\cdots (q)_{s_{k}-s_{k+1}}(q)_{2s_{k+1}}}
$$
\end{prop}
\vskip10pt
The value of $a=1/2$ is nice because the polynomials $H_{2n}(z,1/2|q)$ factor by the 
$q$-binomial theorem.
\begin{equation}
\label{speciala}
H_{2n}(-zq^{1/2},1/2|q)= (qz,1/z;q)_n.
\end{equation}
So we have \cite[(14)]{Bress1}
\begin{thm} 
\label{iterateBress}
[Bressoud's identity]
\begin{equation}
\begin{aligned}
&\frac{H_{2n}(-zq^{1/2},k+3/2|\ q)}{(q;q)_{2n}} \\
&=\sum_{n\ge s_1\ge s_2\ge\cdots\ge s_{k+1}\ge 0}
\frac{q^{s_1^2+\cdots s_{k+1}^2} (qz,1/z;q)_{s_{k+1}}}
{(q)_{n-s_1}(q)_{s_1-s_2}\cdots (q)_{s_{k}-s_{k+1}}(q)_{2s_{k+1}}}
\end{aligned}
\end{equation}
\end{thm}

Taking the $n\to \infty$ limit of Theorem~\ref{iterateBress} using the Jacobi Triple Product identity gives 
\begin{equation}
\label{Hlimit}
\lim_{n\to\infty} H_{2n}(-z,a |q)= \frac{(q^{2a},zq^{a},q^{a}/z;q^{2a})_\infty}{(q)_\infty}.
\end{equation}

We obtain
\begin{cor}
\label{inftylimit}
For a non-negative integer $k$,
$$
\frac{(q^{2k+3},zq^{k+2},q^{k+1}/z;q^{2k+3})_\infty}{(q;q)_\infty}=
\sum_{s_1\ge s_2\ge\cdots\ge s_{k+1}\ge 0}
\frac{q^{s_1^2+\cdots s_{k+1}^2} (qz,1/z;q)_{s_{k+1}}}
{(q)_{s_1-s_2}\cdots (q)_{s_{k}-s_{k+1}}(q)_{2s_{k+1}}}.
$$
\end{cor}

Note that Corollary~\ref{inftylimit} immediately gives 
the Andrews-Gordon identities \eqref{AGiden} for $r=0$. If $z=1$, this 
choice of $z$ forces $s_{k+1}=0.$  
The choice of $z=q^{r}$ does give the right side of the Andrews-Gordon 
identities \eqref{AGiden}, but not the left side. There is an extra sum 
over $s_{k+1}$, and the power of $q$ does not match. 

\begin{Andrques} \label{Qu}
Is there a simple way to understand why the choice of
$z=q^{r}$ eliminates the $s_{k+1}$ sum and replaces it with a power of $q$? 
\end{Andrques}

We now answer Andrews' question, and we will use this answer in subsequent sections. 
The ingredient we need appeared in a paper of Garrett, Ismail and Stanton, \cite{GIS}.

\begin{prop} 
\label{funceq}
For any $c$,
$$
H_{2n}(-q^c,c|q)=q^n H_{2n}(-q^{c-1},c|q).
$$
\end{prop}
 
To answer Andrews' question, start with $H_{2n}(-q^{r+1/2}, k+3/2 |q)$, apply 
Lemma~\ref{keyBrlemma} $k-r+1$ times to obtain $H_{2s_{k-r+1}}(-q^{r+1/2}, r+1/2 |q)$.
Next apply Proposition~\ref{funceq} once to obtain 
$q^{s_{k-r+1}} H_{2s_{k-r+1}}(-q^{r-1/2}, r+1/2 |q)$. This is the linear exponent in $q$ we need. 
The remaining exponents arise from again applying  Lemma~\ref{keyBrlemma} 
followed by Proposition~\ref{funceq}. The final sum on $s_{k+1}$ is now eliminated, because 
the final term becomes $H_{2s_{k+1}}(-q^{1/2}, 1/2 |q)=(q,1;q)_{s_{k+1}}$, which forces $s_{k+1}=0.$
 
\section{New Andrews-Gordon identities}

In this section we prove two new Andrews-Gordon identities for odd moduli. 
The first has binomial factors.

\begin{thm} 
\label{firstnewthm}
For $0\le j,r\le k,$ and $j+r\le k,$
$$
\begin{aligned}
\sum_{ s_1\ge s_2\ge\cdots\ge s_{k}\ge 0}&
\frac{q^{-s_1-\cdots -s_j} (1+q^{s_1+s_2})(1+q^{s_2+s_3})\cdots (1+q^{s_{j-1}+s_j})}
{(q)_{s_1-s_2}\cdots (q)_{s_{k-1}-s_{k}}(q)_{s_{k}}}\\
&\times 
q^{s_{1}^2+\cdots +s_k^2+s_{k-r+1}+\cdots +s_k}\\
&=
\sum_{s=0}^j \binom{j}{s} 
\frac{(q^{k+1-r+j-2s},q^{k+2+r-j+2s},q^{2k+3};q^{2k+3})_\infty}{(q;q)_\infty}.
\end{aligned}
$$
Moreover, the $j$ factors of  $q^{-s_1}$ and  $q^{-s_i}(1+q^{s_{i-1}+s_i})$, 
$2\le i\le j$ may be 
replaced by any $j$-element subset of $\{q^{-s_1}$, $q^{-s_i}(1+q^{s_{i-1}+s_i})\}$, 
$2\le i\le k-r$.
\end{thm}
\vskip10pt\noindent

For example, the binomial factors could occur as the last $j$ of the first $k-r$ summation indices instead of the first $j$ 
indices, namely
$$
\prod_{t=0}^{j-1} q^{-s_{k-r-t}}(1+q^{s_{k-r-1-t}+s_{k-r-t}}).
$$
\vskip10pt
A corollary of Theorem~\ref{firstnewthm} is an identity which contains the Andrews-Gordon identities \eqref{AGiden} 
when $j=0$ and Bressoud's identities \eqref{Bress1} when $r=0$.

\begin{thm} 
\label{secondnewthm}
For $0\le j,r\le k,$ and $j+r\le k,$
$$
\begin{aligned}
\sum_{ s_1\ge s_2\ge\cdots\ge s_{k}\ge 0}&
\frac{q^{s_1^2+\cdots +s_k^2-(s_1+\cdots +s_j)+(s_{k-r+1}+\cdots +s_k)} }
{(q)_{s_1-s_2}\cdots (q)_{s_{k-1}-s_{k}}(q)_{s_{k}}}\\
=&
\sum_{s=0}^j  
\frac{(q^{k+1-r+j-2s},q^{k+2+r-j+2s},q^{2k+3};q^{2k+3})_\infty}{(q;q)_\infty}.
\end{aligned}
$$
\end{thm}
\vskip10pt\noindent

We need a new fact about the Laurent polynomials $H_{2n}(z,a |q).$ 

\begin{prop} 
\label{newprop}
For a non-negative integer $n$,
$$
\frac{H_{2n}(zq,a+1|q)+H_{2n}(q/z,a+1|q)}{(q;q)_{2n}}=
\sum_{s=0}^n \frac{q^{s^2-s}(1+q^{n+s})}{(q;q)_{n-s} }
\frac{H_{2s}(z,a|q)}{(q;q)_{2s}},
$$
\end{prop}
\begin{proof} Note that the left side of Proposition~\ref{newprop} is invariant 
under $z\to 1/z$ so it does have an expansion in terms of $H_{2s}(z,a|q)$. If the 
coefficient of $z^kq^{ak^2}$ is computed for each side, we must show
$$
\left[ \begin{matrix} 2n\\ n-k \end{matrix} \right]_q q^kq^{k^2}+ 
\left[ \begin{matrix} 2n\\ n-k \end{matrix} \right]_q q^{-k}q^{k^2}=
\sum_{s=k}^n \frac{q^{s^2-s} (1+q^{n+s})}{(q)_{2s}} \frac{(q)_{2n}}{(q)_{n-s}}
\left[ \begin{matrix} 2s\\ s-k \end{matrix} \right]_q.
$$
The $s$-sum for the term $q^{n+s}$ is summable as a 
product by a limiting case of the $q$-Vandermonde sum, see 
\cite[p. 354, (II.6)]{GR}.
The $s$-sum for the term $1$ is nearly summable, it is a sum of two products. 
Putting these terms together yields the two terms on the left side. 
The details are not given.
\end{proof}

We need some functions which generalize the $H_{2n}(z,a|q)$.
\vskip5pt
\begin{defn} For a non-negative integer $j$, let
$$
F_n^{(0)}(z,a)=H_{2n}(z,a |q), \qquad  F_n^{(j+1)}(z,a)=F_n^{(j)}(zq,a)+F_n^{(j)}(q/z,a), 
\quad j\ge 0.
$$
\end{defn}

Proposition~\ref{newprop} can be rewritten using these new functions. 
The proof is by induction on $j$. 
\begin{prop} 
\label{newprop2}
For non-negative integers $n$ and $j$,
$$
\frac{F_{n}^{(j+1)}(z,a+1)}{(q;q)_{2n}}=
\sum_{s=0}^n \frac{q^{s^2-s}(1+q^{n+s})}{(q;q)_{n-s} }
\frac{F_{s}^{(j)}(z,a)}{(q;q)_{2s}},
$$
\end{prop}

Iterating Proposition~\ref{newprop2} is the next Proposition.
\begin{prop} 
\label{anotherFsum}
For a non-negative integer $n$ and a positive integer $j$, 
$$
\frac{F_n^{(j)}(z,a)}{(q)_{2n}}= \sum_{n\ge s_1 \ge s_2 \ge \dots \ge s_j\ge 0}
\frac{q^{s_1^2-s_1}(1+q^{n+s_1})}{(q)_{n-s_1}}
\prod_{t=2}^{j}\frac{q^{s_t^2-s_t}(1+q^{s_{t-1}+s_t})}{(q)_{s_{t-1}-s_t}}
\frac{F_{2s_j}^{(0)}(z,a-j)}{(q)_{2s_j}}.
$$
\end{prop}

Finally the functions $F$ also satisfy Lemma~\ref{keyBrlemma} because the $H$ functions do.

\begin{prop} 
\label{Fsum}
For a non-negative integer $n$ and a positive integer $j$, 
$$
\frac{F_{n}^{(j)}(z,a)}{(q;q)_{2n}}=\sum_{s=0}^n \frac{q^{s^2}}{(q;q)_{n-s} }
\frac{F_{s}^{(j)}(z,a-1|q)}{(q;q)_{2s}}.
$$
\end{prop}

Any of these functions may be written as a linear combination of 
$F_n^{(0)}(z,a)=H_{2n}(z,a |q)$.
\begin{prop} 
\label{recurseF}
For any non-negative integer $j$,
$$
F_n^{(j+1)}(z,a)=\sum_{s=0}^j \binom{j}{s} \left( F_n^{(0)}(zq^{j+1-2s},a)+F_n^{(0)}(q^{j+1-2s}/z,a) \right).
$$
\end{prop}
\begin{proof} By induction on $j$ we have 
$$
\begin{aligned}
F_n^{(j+1)}(z,a)= &\sum_{s=0}^{j-1} \binom{j-1}{s} \left( F_n^{(0)}(zq^{j+1-2s},a)+F_n^{(0)}(q^{j-1-2s}/z,a) \right)\\
&+ \sum_{s=0}^{j-1} \binom{j-1}{s} \left( F_n^{(0)}(q^{j+1-2s}/z,a)+F_n^{(0)}(zq^{j-1-2s},a)\right) \\
=& \sum_{s=0}^{j} F_n^{(0)}(zq^{j+1-2s},a) \left( \binom{j-1}{s}+\binom{j-1}{s-1}\right)\\
+& \sum_{s=0}^{j} F_n^{(0)}(q^{j+1-2s}/z,a) \left( \binom{j-1}{s-1}+\binom{j-1}{s}\right)\\
=& \sum_{s=0}^{j} \binom{j}{s}\left( F_n^{(0)}(zq^{j+1-2s},a)+F_n^{(0)}(q^{j+1-2s}/z,a)\right).
\end{aligned}
$$
\end{proof}

We have two expressions for $F_n^{(j)}(z,a)$: Propositions~\ref{recurseF} and    
~\ref{anotherFsum}. The proof of Theorem~\ref{firstnewthm} uses these two 
expressions after taking a limit as $n\to\infty.$
We record the appropriate $n\to \infty$ limit of Proposition~\ref{recurseF}.

\begin{prop} 
\label{Flimit}
If $j$ is a non-negative integer, 
$$
\lim_{n\to\infty} F_n^{(j+1)}(-z,a)=\frac{1}{(q)_\infty} \sum_{s=0}^{j+1} \binom{j+1}{s} 
(q^{2a},zq^{a+j+1-2s},q^{a-j-1+2s}/z;q^{2a})_\infty.
$$
\end{prop}
\begin{proof}
Applying \eqref{Hlimit} and Proposition~\ref{recurseF} we have 
$$
\begin{aligned}
\lim_{n\to\infty} F_n^{(j+1)}(-z,a)=&\frac{1}{(q)_\infty} \sum_{s=0}^j \binom{j}{s}\bigl(
(q^{2a},zq^{a+j+1-2s},q^{a-j-1+2s}/z;q^{2a})_\infty\\
&+(q^{2a},q^{a+j+1-2s}/z,zq^{a-j-1+2s};q^{2a})_\infty\bigr)\\
=&\frac{1}{(q)_\infty} \sum_{s=0}^{j+1} \left(\binom{j}{s}+\binom{j}{j+1-s}\right) 
(q^{2a},zq^{a+j+1-2s},q^{a-j-1+2s}/z;q^{2a})_\infty\\
=&\frac{1}{(q)_\infty} \sum_{s=0}^{j+1} \binom{j+1}{s} 
(q^{2a},zq^{a+j+1-2s},q^{a-j-1+2s}/z;q^{2a})_\infty.
\end{aligned}
$$
\end{proof}

\vskip10pt
\begin{proof}[Proof of Theorem~\ref{firstnewthm}] We see from 
Proposition~\ref{Flimit} that the right side of Theorem~\ref{firstnewthm} is
$$
\lim_{n\to\infty} F_n^{(j)}(-z,k+3/2), \qquad z=q^{-1/2-r}.
$$
or at $z=q^{r+1/2}$ since all functions are symmetric under $z\to 1/z.$
We apply Proposition~\ref{Fsum} to obtain $j$ sums and a factor of 
$$
\frac{F_{2s_j}^{(0)}(-z,k+3/2-j)}{(q)_{2s_j}}=\frac{H_{2s_j}(-z,k+3/2-j|q)}{(q)_{2s_j}}.
$$ 
Now we are in the realm of the 
Andrews-Gordon proof in section 2. We finish the proof as before, by 
applying Lemma~\ref{keyBrlemma} $k-r$ times, and then inserting 
the linear factors $r$ times. 

Since the functions $F_n^{(j)}(z,a)$ also satisfy Proposition~\ref{Fsum}, we could 
apply Proposition~\ref{Fsum} anytime before we use 
Proposition~\ref{newprop2} in the first $k-r$
iterates. This gives the arbitrary choice of the binomials. 
\end{proof}

Next we derive Theorem~\ref{secondnewthm} from Theorem~\ref{firstnewthm}. The idea is take an appropriate linear combination of Theorem~\ref{firstnewthm} to replace the binomial 
factors in  Theorem~\ref{firstnewthm} by a single term $q^{-s_1-s_2-\dots -s_j}.$ 
For example if $j=3$,
\begin{equation}
\label{weirdedges}
q^{-s_1-s_2-s_3}(1+q^{s_1+s_2})(1+q^{s_2+s_3})-q^{-s_3}(1+q^{s_2+s_3})-q^{-s_1}=
q^{-s_1-s_2-s_3}
\end{equation}
yields, for the right side of Theorem~\ref{firstnewthm},
$$
\begin{aligned}
&\sum_{s=0}^3 \binom{3}{s} 
\frac{(q^{k+1-r+3-2s},q^{k+2+r-3+2s},q^{2k+3};q^{2k+3})_\infty}{(q;q)_\infty}\\
&-2\sum_{s=0}^1 \binom{1}{s} 
\frac{(q^{k+1-r+1-2s},q^{k+2+r-1+2s},q^{2k+3};q^{2k+3})_\infty}{(q;q)_\infty}\\
=&\sum_{s=0}^3  
\frac{(q^{k+1-r+1-2s},q^{k+2+r-1+2s},q^{2k+3};q^{2k+3})_\infty}{(q;q)_\infty}
\end{aligned}
$$
as predicted by Theorem~\ref{secondnewthm}.

The version of \eqref{weirdedges} we need for general $j$ uses edges in a graph which 
is a path from $1$ to $j$: $1-2-3-\dots-j.$ A pair of edges in this 
graph do not overlap if they do not share a vertex. For a set $E$ of non-overlapping edges let
$$
wt(E)=\prod_{i\notin E, i\ge 2} q^{-s_i}(1+q^{s_{i-1}+s_i}) \times 
\begin{cases} q^{-s_1} \quad {\text{if }} 1\notin E\\
1 \quad {\text{if }} 1\in E.
\end{cases}
$$
Here are the three possible sets of non-overlapping edges $E$ for $j=3$,
$$
\begin{aligned}
E=&\varnothing , \quad  wt(E)= q^{-s_1-s_2-s_3}(1+q^{s_1+s_2})(1+q^{s_2+s_3}),\\
E=&1-2, \quad  wt(E)=q^{-s_3}(1+q^{s_2+s_3}),\\
E=&2-3, \quad  wt(E)=q^{-s_1}.
\end{aligned}
$$ 
These are the three terms in \eqref{weirdedges}.

\begin{lem} 
\label{edgelemma}
We have
$$
\sum_{E} (-1)^{|E|} wt(E)=q^{-s_1-s_2-\dots -s_j},
$$
where the sum is over all non-overlapping edge sets $E$ of $1-2-3-\dots-j.$
\end{lem}

\begin{proof} Again we do an induction on $j$. Suppose $E$ is a set of 
non-overlapping edges for $1-2-3-\dots-(j+1).$ If the last edge $j-(j+1)$ 
is in $E$, 
the remaining edges are non-overlapping for $j-1$, so by induction
$$
\sum_{E, j-(j+1)\in E} (-1)^{|E|} wt(E)=- q^{-s_1-s_2-\dots -s_{j-1}}.
$$
If the last edge $j-(j+1)$ 
is not in $E$, 
the remaining edges are non-overlapping for $j$, so by induction
$$
\sum_{E, j-(j+1)\notin E} (-1)^{|E|} wt(E)=q^{-s_{j+1}}(1+q^{{s_j}+s_{j+1}}) q^{-s_1-s_2-\dots -s_{j}}.
$$
Because
$$
- q^{-s_1-s_2-\dots -s_{j-1}}+q^{-s_{j+1}}(1+q^{{s_j}+s_{j+1}}) q^{-s_1-s_2-\dots -s_{j}}=
q^{-s_1-s_2-\dots -s_{j+1}}
$$
we are done. 
\end{proof}

\begin{proof}[Proof of Theorem~\ref{secondnewthm}] It remains to show that 
the linear combination given by Lemma~\ref{edgelemma} gives the correct constants 
for the infinite products on the right side of 
Theorem~\ref{secondnewthm} (namely 1). 

There are $\binom{j-t}{t}$ such $E$ with $t$ edges, and 
therefore $j-2t$ vertices not in $E$. So the right side becomes
$$
\sum_{t=0}^{[j/2]} \binom{j-t}{t} (-1)^t
\sum_{s=0}^{j-2t} \binom{j-2t}{s}
\frac{(q^{k+1-r+j-2t-2s},q^{k+2+r-j+2t+2s},q^{2k+3};q^{2k+3})_\infty}{(q;q)_\infty}
$$
The coefficient of
$$
\frac{(q^{k+1-r+j-2u},q^{k+2+r-j+2u},q^{2k+3};q^{2k+3})_\infty}{(q;q)_\infty}
$$ 
for $0\le u\le j$ is
$$
\sum_{s,t,s+t=u} \binom{j-t}{t} (-1)^t \binom{j-2t}{s} =1
$$
by the terminating Chu-Vandermonde evaluation.
\end{proof}

\section{New Bressoud Identities}
The Bressoud identities for even moduli can be proven the same way. 
The only change is to replace \eqref{speciala} by
$$
\frac{H_{2s}(-1,1|q)}{(q)_{2s}}=\frac{1}{(q^2;q^2)_s}.
$$

Here we state (without proof) the analogous binomial results for even moduli.

\vskip10pt\noindent
\begin{thm} For $0\le j,r\le k,$ and $j+r\le k,$
$$
\begin{aligned}
\sum_{ s_1\ge s_2\ge\cdots\ge s_{k}\ge 0}&
\frac{q^{-s_1-\cdots -s_j} (1+q^{s_1+s_2})(1+q^{s_2+s_3})\cdots (1+q^{s_{j-1}+s_j})}
{(q)_{s_1-s_2}\cdots (q)_{s_{k-1}-s_{k}}(q^2;q^2)_{s_{k}}}\\
&\times 
q^{s_{1}^2+\cdots +s_k^2+s_{k-r+1}+\cdots +s_k}\\
&=
\sum_{s=0}^j \binom{j}{s} 
\frac{(q^{k+1-r+j-2s},q^{k+1+r-j+2s},q^{2k+2};q^{2k+2})_\infty}{(q;q)_\infty}.
\end{aligned}
$$
Moreover, the $j$ factors of  $q^{-s_1}$ and  $q^{-s_i}(1+q^{s_{i-1}+s_i})$, 
$2\le i\le j$ may be 
replaced by any $j$-element subset of $\{q^{-s_1}$, $q^{-s_i}(1+q^{s_{i-1}+s_i})$, 
$2\le i\le k-r\}$.
\end{thm}

\vskip10pt\noindent
Again using Lemma~\ref{edgelemma} we have the version without binomial 
coefficients.

\begin{thm} 
\label{secondeventhm}
For $0\le j,r\le k,$ and $j+r\le k,$
$$
\begin{aligned}
\sum_{ s_1\ge s_2\ge\cdots\ge s_{k}\ge 0}&
\frac{q^{s_1^2+\cdots +s_k^2-(s_1+\cdots +s_j)+(s_{k-r+1}+\cdots +s_k)} }
{(q)_{s_1-s_2}\cdots (q)_{s_{k-1}-s_{k}}(q^2;q^2)_{s_{k}}}\\
=&
\sum_{s=0}^j  
\frac{(q^{k+1-r+j-2s},q^{k+1+r-j+2s},q^{2k+2};q^{2k+2})_\infty}{(q;q)_\infty}.
\end{aligned}
$$
\end{thm}
\vskip10pt\noindent
The case $j=0$ in Theorem~\ref{secondeventhm} is \cite[(3.4), p. 15]{Bress2} 
while $r=0$ is \cite[(3.5), p. 16]{Bress2}.

\vskip10pt\noindent

\section{Overpartitions}

Finally, for completeness, we give two analogous results for overpartitions,
see \cite{And2}. Proposition~\ref{funceq}, which is used to insert linear exponents, 
requires a special choice of $z$. But in the two results of this section 
we have a general $z$, so we cannot insert the linear factors as before.

The first result is a binomial version of Corollary~\ref{inftylimit}.

\begin{thm}
\label{overpart1} 
For $0\le j\le k+1,$
$$
\begin{aligned}
\sum_{ s_1\ge s_2\ge\cdots\ge s_{k+1}\ge 0}&
\frac{q^{-s_1-\cdots-s_j} (1+q^{s_1+s_2})(1+q^{s_2+s_3})\cdots (1+q^{s_{j-1}+s_j})}
{(q)_{s_1-s_2}\cdots (q)_{s_{k}-s_{k+1}}}
\frac{(-z,-q/z;q)_{s_{k+1}}}{(q)_{2s_{k+1}}}\\
&\times 
q^{s_{1}^2+\cdots +s_{k+1}^2}\\
&=
\sum_{s=0}^j \binom{j}{s} 
\frac{(-zq^{k+1+j-2s},-q^{k+2-j+2s}/z,q^{2k+3};q^{2k+3})_\infty}{(q;q)_\infty}.
\end{aligned}
$$
Moreover, the $j$ factors of  $q^{-s_1}$ and  $q^{-s_i}(1+q^{s_{i-1}+s_i})$, 
$2\le i\le j$ may be 
replaced by  any $j$-element subset of $\{q^{-s_1}$, $q^{-s_i}(1+q^{s_{i-1}+s_i})$, 
$2\le i\le k+1\}$.
\end{thm}

\vskip10pt\noindent
\begin{thm} 
\label{overpart2}
For $0\le j\le k+1,$
$$
\begin{aligned}
\sum_{ s_1\ge s_2\ge\cdots\ge s_{k+1}\ge 0}&
\frac{q^{s_1^2-s_1+\cdots +s_j^2-s_j+s_{j+1}^2+\cdots +s_{k+1}^2}}
{(q)_{s_1-s_2}\cdots (q)_{s_{k}-s_{k+1}}}
\frac{(-z,-q/z;q)_{s_{k+1}}}{(q)_{2s_{k+1}}}\\
&=
\sum_{s=0}^j 
\frac{(-zq^{k+1+j-2s},-q^{k+2-j+2s}/z,q^{2k+3};q^{2k+3})_\infty}{(q;q)_\infty}.
\end{aligned}
$$
\end{thm}

\vskip10pt\noindent
We mention a different expansion for the infinite product in 
Theorem~\ref{overpart2} when $j=0.$ This final result comes from a version of 
the Laurent polynomials $H_{2n}(z,a|q)$ with an odd index. We do not develop 
the corresponding results here.   

\begin{thm}
\label{overpart3} 
If $k$ is a non-negative integer, 
$$
\begin{aligned}
\sum_{ s_1\ge s_2\ge\cdots\ge s_{k+1}\ge 0}&
\frac{q^{s_1^2+\cdots +s_{k+1}^2+s_1+\dots+s_{k+1}}}
{(q)_{s_1-s_2}\cdots (q)_{s_{k}-s_{k+1}}}
\frac{(-q^{k+1}/z)_{s_{k+1}+1}(-zq^{-k})_{s_{k+1}}}{(q)_{2s_{k+1}+1}}\\
&=
\frac{(-zq^{k+1},-q^{k+2}/z,q^{2k+3};q^{2k+3})_\infty}{(q;q)_\infty}.
\end{aligned}
$$
\end{thm}

If $k=0$ in Theorems~\ref{overpart2} and \ref{overpart3}, we have the curious result
(see \cite[(5.1)]{And2})
\begin{equation}
\begin{aligned}
\frac{(-zq,-q^2/z,q^3;q^3)_\infty}{(q)_\infty}=&
\sum_{s=0}^\infty \frac{(-z,-q/z;q)_s}{(q)_{2s}}q^{s^2} \\
=&
\sum_{s=0}^\infty \frac{(-zq)_{s+1}(-1/z)_s}{(q)_{2s+1}}q^{s^2+s}. 
\end{aligned}
\end{equation}

\section{Remarks}

The Andrews-Gordon identities have combinatorial interpretations for integer partitions,  
three of which are (see \cite{And1}): 
\begin{enumerate}
\item  those with modular conditions on parts,
\item those with difference conditions on parts,
\item those with conditions on iterated Durfee squares.
\end{enumerate} 

This paper offers no insightful versions of these results for the 
binomial versions given here.

Berkovich and Paule \cite{BP1},\cite{BP2} have versions of the Andrews-Gordon identities where 
the linear forms are also modified.

Griffin, Ono, and Warnaar \cite{GOW} give new infinite families 
(e.g. \cite[Theorem 1.1]{GOW}) of Rogers-Ramanujan identities. 
See \cite[(2.7)]{GOW} for the Andrews-Gordon-Bressoud identities in their paper.

Seo and Yee \cite{SY} combinatorially study singular overpartitions, 
whose generating function is given by $j=0$ in 
Theorem~\ref{overpart1} with a special choice of $z$.

\end{document}